\documentclass[10pt]{article}

\usepackage[affil-it]{authblk}
\usepackage{amsfonts,amsmath,amssymb,amsthm,amscd,mathtools}
\usepackage[dvipsnames]{xcolor}
\usepackage[margin=2.6cm]{geometry}
\usepackage{graphicx,caption,float,subcaption}
\graphicspath{{./}}
\usepackage{epstopdf,epsfig}
\usepackage{bm}
\usepackage{parskip}
\usepackage{cases}
\usepackage{colortbl,hhline}
\usepackage{booktabs}
\usepackage{enumerate}
\usepackage[shortlabels]{enumitem}
\usepackage[hidelinks]{hyperref}

\newtheorem{theorem}{Theorem}[section]
\newtheorem{lemma}[theorem]{Lemma}
\newtheorem{corollary}[theorem]{Corollary}

\theoremstyle{definition}
\newtheorem{definition}{Definition}
\newtheorem{remark}{Remark}
\newtheorem{gltax}{GLT}

\newcommand{\N}{\mathbb{N}}
\newcommand{\Z}{\mathbb{Z}}
\newcommand{\R}{\mathbb{R}}
\newcommand{\C}{\mathbb{C}}
\newcommand{\iu}{\mathrm{i}}
\newcommand{\glt}{\mathrm{GLT}}
\DeclareMathOperator*{\rk}{rank}
\DeclareMathOperator*{\diag}{diag}
\newcommand\norm[1]{\| #1 \|}
\newcommand{\abs}[1]{\left| #1 \right|}

\title{Determining the space dependent coefficients in space-time fractional diffusion equations via Krylov preconditioning}

\author[1]{Asim Ilyas\thanks{ailyas1@uninsubria.it}}
\author[1]{Muhammad Faisal Khan\thanks{mfkhan@uninsubria.it}}
\author[2,3]{Rosita L. Sormani\thanks{rositaluisa.sormani@uninsubria.it}}
\author[1]{Giacomo Tento\thanks{gtento@uninsubria.it}}
\author[2,4]{Stefano Serra-Capizzano\thanks{s.serracapizzano@uninsubria.it}}

\affil[1]{Department of Theoretical and Applied Sciences, University of Insubria, Varese, Italy}
\affil[2]{Department of Science and High Technology, University of Insubria, Como, Italy}
\affil[3]{Department of Mathematics and Computer Science, University of Cagliari, Cagliari, Italy}
\affil[4]{Division of Scientific Computing, Department of Information Technology, Uppsala University, Uppsala, Sweden}

\date{}

\begin{document}
\maketitle

\begin{abstract}
    We consider a time-space fractional diffusion equation with a variable coefficient and investigate the inverse problem of reconstructing the source term, after regularizing the problem with the quasi-boundary value method to mitigate the ill-posedness. The equation involves a Caputo fractional derivative in the space variable and a tempered fractional derivative in the time variable, both of order in $(0,1)$. A finite difference approximation leads to a two-by-two block linear system of large dimensions. We conduct a spectral analysis of the associated matrix sequences, employing tools from Generalized Locally Toeplitz (GLT) theory, and construct the preconditioner guided by the GLT analysis. Numerical experiments are reported and commented, followed by concluding remarks.
\end{abstract}

\section{Introduction}

The aim of this work is to numerically solve the inverse problem of identifying the source term $f(x)$ in the space-time fractional diffusion equation
\begin{equation} \label{eq:problem}
    \begin{dcases}
        {}^T\!D_t^{\xi,\rho} u(x,t) + a(x) {}^C\!D_x^{\eta} u(x,t) = q(t)f(x),
        & x\in (0,1), \; t\in (0,T),
        \\
        u(0,t) = u(1,t) = 0,
        & t\in (0,T),
        \\
        u(x,0) = \varphi(x),
        & x\in (0,1),
        \\
        u(x,T) = \phi(x),
        & x\in (0,1),
    \end{dcases}
\end{equation}
where $T>0$ is the final time, $a(x)$ is a continuous a.e.\! variable coefficient such that $a(x)\neq 0$, $q(t)$ is a given function, $\varphi(x)$ is the initial condition, and $\phi(x)$ is the final time data. Moreover, ${}^{T}\!D_{t}^{\xi,\rho}$ represents the tempered fractional derivative in the time variable of order $\xi\in(0,1)$, defined for $\rho>0$ as follows
\begin{equation*}
    {}^T\!D_t^{\xi,\rho} u(x,t) := \frac{1}{\Gamma(1-\xi)} \int_0^t e^{-\rho(t-\tau)} (t-\tau)^{-\xi} \frac{\partial}{\partial\tau} u(x,\tau) \,\mathrm{d} \tau,
\end{equation*}
while ${}^C\!D_x^{\eta}$ denotes the Caputo fractional derivative in the space variable of order $\eta\in(0,1)$, defined as
\begin{equation*}
    {}^C\!D^\eta_x u(x,t) := \frac{1}{\Gamma(1-\eta)} \int_0^x {(x-\zeta)^{-\eta}} \frac{\partial}{\partial\zeta} u(\zeta,t) \,\mathrm{d}\zeta.
\end{equation*}

The motivation for this work relies on the wide popularity gained by Fractional Diffusion Equations (FDEs) in recent years, since they constitute powerful tools for modeling complex phenomena across various scientific disciplines. Thanks to the inclusion of fractional derivatives, FDEs can effectively capture memory effects, non-local interactions, and anomalous behaviors, making them extremely valuable in many applicative fields, such as viscoelasticity \cite{Zhou2022}, anomalous diffusion \cite{Kilbas2006}, biology \cite{Magin2010,Owolabi2024}, chemistry \cite{Scher2002}, finance \cite{Scalas2000}, and chaotic systems \cite{Wang2011}. In particular, the Caputo fractional derivative is widely used in mathematical modeling due to its ability to capture the impact of nonlocal spatial heterogeneities and incorporate physically meaningful initial conditions while avoiding singularities at zero. Moreover, tempered fractional derivatives are particularly useful for modeling transient anomalous diffusion, where diffusion initially follows a fractional power law but eventually transitions into a normal diffusion profile at long times \cite{Sandev2017}. Equation \eqref{eq:problem} incorporates tempered memory effects in time and fractional spatial diffusion, making it highly suitable for modeling complex transport phenomena, including subdiffusion, superdiffusion, and anomalous dispersion in heterogeneous media.

Furthermore, inverse problems arise in many practical applications where one seeks to determine unknown causes from observed effects. However, these problems are often ill-posed and hence, in particular, small perturbations in the data can lead to significant errors in the solution. As a consequence, various regularization techniques have been developed to ensure well-posedness, and in particular inverse problems for FDE have been studied through diverse analytical and computational approaches \cite{Rooh2024,Malik2024,Ilyas2024,Jin2015,Suhaib2023, Tuan2017}. 

To address the ill-posedness in the sense of Hadamard of the inverse problem \eqref{eq:problem}, we apply the so-called nonlocal or quasi-boundary value method, transforming it into a well-posed formulation. This technique has been successfully applied to handle parabolic and fractional diffusion equations \cite{Denche2005,Hao2008}. First, we observe that the measurement of the final time data is inevitably contaminated by noise, therefore we denote the noisy data as $\phi_\varepsilon(x)$, where $\varepsilon>0$ is the noise level, and assume that
\begin{equation*}
    \norm{\phi - \phi_\varepsilon}_{L^2} \leq \varepsilon,
\end{equation*}
where $\norm{\cdot}_{L^2}$ is the $L^2$ norm on $(0,1)$. Applying the quasi-boundary value method leads to the following well-posed regularized problem
\begin{equation} \label{eq:quasi-problem}
    \begin{dcases}
        {}^T\!D_t^{\xi,\rho} u(x,t) + a(x) {}^C\!D_x^{\eta} u(x,t) = q(t)f_{\varepsilon,\lambda}(x),
        & x\in (0,1), \; t\in (0,T),
        \\
        u(0,t) = u(1,t) = 0,
        & t\in (0,T),
        \\
        u(x,0) = \varphi(x),
        & x\in (0,1),
        \\
        u(x,T) = \phi_\varepsilon(x) - \lambda\, a(x) {}^C\!D_x^{\eta} f_{\varepsilon,\lambda}(x),
        & x\in (0,1),
    \end{dcases}
\end{equation}
where $\lambda>0$ is a regularization parameter, $f_{\varepsilon,\lambda}$ approximates $f(x)$ and the solution $u(x,t)$ represents an approximation of the solution of the original problem \eqref{eq:problem}.

We aim to develop an efficient algorithm for reconstructing the space-dependent source term $f(x)$, starting from the regularized equation \eqref{eq:quasi-problem}. Many efficient numerical methods have been developed to solve FDEs, including finite differences \cite{Chen2014,Hao2021,Thomas2013}, finite elements \cite{Bu2014,Dhatt2012,Li2018}, finite volumes \cite{Eymard2000,Simmons2017}, and spectral methods \cite{Gottlieb1977}. However, high precision often requires very fine spatial and time discretizations, leading to large and complex systems of equations and expensive computations, especially considering that the intrinsic nonlocal character of fractional differential operators leads to dense matrices. Nevertheless, when FDEs are discretized using a uniform grid, the resulting coefficient matrices usually have a Toeplitz-like structure, which allows for faster and more efficient solvers, such as preconditioned iterative solvers \cite{Donatelli2016,Lei2013,Pang2021}.

Here, we discretize the continuous problem through a finite difference scheme, which leads to a two-by-two block linear system, and focus our attention on the design of an effective preconditioner for the Generalized Minimal RESidual (GMRES) method \cite{Saad1986GMRES}. Since the preconditioner needs to approximate the original matrix in a spectral sense to be efficient, we first perform an asymptotical spectral analysis of the coefficient matrix sequences. We use the theory of Generalized Locally Toeplitz (GLT) matrix sequences, which is needed due to presence of the variable coefficient $a(x)$, which spoils the Toeplitzness of the involved structures; see \cite{garoni2017,garoni2018,barbarino2020uni,barbarino2020multi} for a complete treatment of the theory and the tutorials \cite{tom,glt-tutorial} for a gentle guide on the use of the theory in practical applications. The spectral study is then employed for developing an appropriate preconditioning strategy and for evaluating the approximation quality of the preconditioning matrix sequence with respect to the coefficient matrix sequence.

The paper is organized as follows. Section \ref{sec:prelim} introduces the theoretical framework and provides a summary of the essential tools. Section \ref{sec:discre} provides the discretization of the regularized problem by using a finite difference scheme in space and time and presents the resulting linear system, emphasizing the structural features of the resulting matrices. In Section \ref{sec:glt-analysis} the coefficient matrix sequence is studied via GLT tools, while in Section \ref{sec:precond} the same type of analysis is performed for the preconditioned matrix sequences. In Section \ref{sec:num}, numerical experiments are conducted to assess the effectiveness of the proposed preconditioner. Concluding remarks are presented in Section \ref{sec:end}.

\section{Preliminaries} \label{sec:prelim}

We briefly introduce the fundamental tools employed in this work, which are discussed in detail in \cite{garoni2017}. For simplicity, we restrict the definitions and theorems to the unilevel scalar case, which is sufficient for the purposes of the current paper. However, all the concepts can be extended the multilevel \cite{garoni2018} and block \cite{barbarino2020uni,barbarino2020multi} settings, that arise from more complex differential problems such as the multi-dimensional extension of \eqref{eq:problem}.


Let us first establish some useful notation.
\begin{itemize}
    \item Given a square matrix $A_n$, we denote with $\lambda_j(A_n)$ and $\sigma_j(A_n)$ the $j$-th eigenvalue and singular value of $A_n$, respectively, as $j=1,\ldots,n$. 
    \item The spectral norm for a square matrix $A_n$ is denoted as $\norm{A_n}$.
    \item Whenever we use terminology from measure theory, for instance “measurable set”, “measurable function”, “a.e.”, we always refer to the Lebesgue measure in $\R^t$, denoted with $\mu_t$.
\end{itemize}

\subsection{Spectral tools}

Throughout the paper, we deal with matrix sequences, referring to any sequence of the form $\{A_n\}_n$, where $A_n$ is a square matrix of size $d_n$ and $d_n \to \infty$ as $n\to\infty$. To a given matrix sequence, it is often possible to associate a spectral or singular value distribution, according to the following definition.

\begin{definition} \label{def:distrib}
    Let $\{A_n\}_n$ be a matrix sequence, where $A_n$ has size $d_n\times d_n$, and let $\psi:D\subset\R^t \to\C$ be a measurable function defined on a set $D$ with $0 < \mu_t(D) < \infty$. Denote with $C_c(\mathbb{K})$ the set of continuous complex-valued functions with bounded support on $\mathbb{K}\in\{\R,\C\}$.
    \begin{itemize}
        \item $\{A_n\}_n$ has an \emph{(asymptotic) singular value distribution} described by $\psi$ if for any $F\in C_c(\R)$
        \begin{equation*}
            \lim_{n\to\infty} \frac{1}{d_n} \sum_{i=1}^{d_n} F  \big(\sigma_i(A_n)\big) = \frac{1}{\mu_t(D)} \int_D F\big(\abs{\psi(\mathbf{x})}\big) \,\mathrm{d}\mathbf{x},
        \end{equation*}
        where $\sigma_i(A_n)$, $i=1,\ldots,d_n$, are the singular values of $A_n$. We use the notation $\{A_n\}_n \sim_\sigma \psi$.
        \item $\{A_n\}_n$ has an \emph{(asymptotic) spectral} or \emph{eigenvalue distribution} described by $\psi$ if for any $F\in C_c(\C)$
        \begin{equation*}
            \lim_{n\to\infty} \frac{1}{d_n} \sum_{i=1}^{d_n} F  \big(\lambda_i(A_n)\big) = \frac{1}{\mu_t(D)} \int_D F\big(\psi(\mathbf{x})\big) \,\mathrm{d}\mathbf{x},
        \end{equation*}
        where $\lambda_i(A_n)$, $i=1,\ldots,d_n$, are the singular values of $A_n$. We use the notation $\{A_n\}_n \sim_\lambda \psi$.
    \end{itemize}
\end{definition}

The informal meaning behind Definition \ref{def:distrib} is the following: if $\psi$ is continuous a.e.\! and $n$ is large, then the eigenvalues or singular values of $A_n$ (suitably ordered and possibly except for $o(d_n)$ outliers) are approximated by a uniform sampling of $\psi$ over its domain.

The notion of clustering can be seen a special case of distribution. In what follows, we denote the $\varepsilon$-expansion of a set $S\subseteq\C$ as
\begin{equation*}
    B(S,\varepsilon) := \bigcup_{z\in S} B(z,\varepsilon),
\end{equation*}
where $B(z,\varepsilon) := \{w\in\C : \abs{w-z} < \varepsilon\}$ is the complex disk with center $z$ and radius $\varepsilon > 0$. 

\begin{definition} \label{clustering}
    Let $\{A_n\}_n$ be a matrix sequence, with $A_n$ of size $d_n\times d_n$, and let $S\subseteq\C$ be nonempty and closed. The sequence $\{A_n\}_n$ is \emph{strongly clustered} at $S$ in the sense of the eigenvalues if $\forall\varepsilon >0$
    \begin{equation*}
        \# \{ j\in\{1,\ldots,d_n\} : \lambda_j(A_n) \notin B(S, \varepsilon) \} = O(1),
        \qquad n\to\infty.
    \end{equation*}
    In other words, the number of eigenvalues of $A_n$ outside $B(S,\varepsilon)$ is bounded by a constant independent of $n$. Moreover, $\{A_n\}_n$ is \textit{weakly clustered} at $S$ if $\forall\varepsilon >0$
    \begin{equation*}
        \# \{ j\in\{1,\ldots,d_n\} : \lambda_j(A_n) \notin B(S, \varepsilon) \} = o(d_n),
        \qquad n\to\infty,
    \end{equation*}
    meaning that the number of eigenvalues of $A_n$ outside $B(S,\varepsilon)$ is negligible with respect to the size of the matrix. A corresponding definition can be given for the singular values, with $S\subseteq\R^+$.
\end{definition} 

\begin{remark} \label{remark:distrib-cluster}
    Let us clarify the relationship between the concepts of distribution and clustering. We recall that the essential range of a measurable function $g:D\subseteq\R^t\to\C$ is the set
    \begin{equation*}
        \{ z\in\C : \mu_t \big(\{ g \in B(z, \varepsilon) \} \big)> 0 \quad\forall\varepsilon > 0 \}.
    \end{equation*}
    As reported in \cite[Theorem 4.2]{Golinskii2007}, it holds
    \begin{equation*}
        \{A_n\}_n \sim_\lambda \psi
        \qquad\Longrightarrow\qquad
        \{A_n\}_n \text{ is weakly clustered at the essential range of } \psi.
    \end{equation*}
    Furthermore, if the essential range of $\psi$ consists of a single number $s\in\C$, then
    \begin{equation*}
        \{A_n\}_n \sim_\lambda \psi
        \qquad\iff\qquad
        \{A_n\}_n \text{ is weakly clustered at } s \text{ in the sense of the eigenvalues}.
    \end{equation*}
    The latter case is typically of interest in the context of preconditioning. Corresponding statements can be given for the singular values, with the obvious suitable changes.
\end{remark}

\subsection{The GLT \texorpdfstring{$\ast$}{}-algebra} 

A GLT sequence is a matrix sequence belonging to the $\ast$-algebra generated by the three specific classes of matrix sequences: zero-distributed, Toeplitz and diagonal sampling matrix sequences. We define them in the following paragraphs and they can be seen as the basic building blocks of the GLT $\ast$-algebra.

Each GLT sequence is equipped with a measurable function $\kappa :[0,1]^d \times[-\pi,\pi]^d \to \C$, with $d\geq 1$, called the GLT \emph{symbol}. When $d=1$, we are in the context of unilevel GLT sequences, while $d>1$ relates to the multilevel case. The symbol is essentially unique, in the sense that if $\kappa, \varsigma$ are two symbols of the same  GLT sequence, then $\kappa = \varsigma$ a.e. We write \( \{A_n\}_n \sim_{\mathrm{GLT}} \kappa \) to denote that $\{A_n\}_n$ is a GLT 
sequence with symbol $\kappa$.

\paragraph{Zero-distributed sequences.}
A matrix sequence $\{Z_n\}_n$ is zero-distributed if it holds $\{Z_n\}_n \sim_\sigma 0$. The following Lemma provides a practical characterization.

\begin{lemma}[{\cite[Theorem 3.2]{garoni2017}}] \label{lemma:zero-distr}
    Let $\{Z_n\}_n$ be a matrix sequence, with $Z_n$ of size $d_n\times d_n$. Then, $\{Z_n\}_n \sim_\sigma 0$ if and only if for every $n\in\N$ it holds $Z_n = R_n + N_n$ with 
    \begin{equation*}
        \lim_{n\to\infty} \frac{\rk(R_n)}{d_n} =  \lim_{n\to\infty} \norm{N_n}\to 0 = 0.
    \end{equation*}
\end{lemma}

\paragraph{Toeplitz sequences.}
For any $n\in\N$, a matrix of the form
\begin{equation*}
    [a_{i-j}]_{i,j = 1}^n \in\C^{n\times n},
\end{equation*}
with coefficients $a_k\in\C$ for $k=1-n,\ldots,n-1$, is called a (unilevel) Toeplitz matrix and it is characterized by constant elements along the diagonals. Given a complex-valued function $f\in L^1\big([-\pi,\pi]\big)$, the $n$-th Toeplitz matrix associated with $f$ is defined as
\begin{equation*}
    T_n(f) := [\hat{f}_{i-j}]_{i,j = 1}^n \in\C^{n\times n},
\end{equation*}
where
\begin{equation*}
    \hat{f}_k = \frac{1}{2\pi}
    \int_{[-\pi,\pi]} f(\theta) e^{-\iu k \theta} \;\mathrm{d}\theta,
    \qquad k\in\Z, \quad \iu^2=-1.
\end{equation*}
are the Fourier coefficients of $f$. The family $\{T_n(f)\}_{n\in\N}$ is the sequence of Toeplitz matrices associated with $f$, called the generating function of the sequence.

We recall that the Wiener class is the sub-algebra of the continuous $2\pi$-periodic functions on $[-\pi,\pi]$ and it corresponds to the set of functions whose Fourier series is absolutely convergent, i.e.,
\begin{equation*}
    \bigg\{ f:[-\pi,\pi]\rightarrow\C \;\colon
    \sum_{k\in\Z} \abs{\hat{f}_k}<+\infty \bigg\}.
\end{equation*}

\paragraph{Diagonal sampling sequences.}
Given $a : [0,1]\to\C$ continuous a.e., for any $n\in\N$ we define the $n$-th diagonal sampling matrix $D_n(a)$ as
\begin{equation*}
    D_n(a) = \diag_{i=1,\ldots,n}
    a\left(\frac{i}{n}\right)
    \in\C^{n\times n}.
\end{equation*}
The family $\{D_n(a)\}_{n\in\N}$ is the sequence of diagonal sampling matrices generated by $a(x)$.

The GLT class satisfies several algebraic and topological properties that are treated in great detail and generality in \cite{barbarino2020uni,barbarino2020multi,garoni2017,garoni2018}. For the purposes of this work, it is sufficient to present them through their operative properties, listed below in the unilevel setting i.e. for $d=1$.

\begin{gltax} \label{glt:distrib}
    If $\{A_n\}_n \sim_\glt \kappa$, then $\{A_n\}_n \sim_\sigma \kappa$ in the sense of Definition \ref{def:distrib}. If moreover each $A_n$ is Hermitian, then $\{A_n\}_n \sim_\lambda \kappa$.
\end{gltax}

\begin{gltax} \label{glt:symbols}
    It holds
    \begin{itemize}
        \item $\{T_n(f)\}_n \sim_\glt \kappa(x,\theta) = f(\theta)$ for any $f\in L^1\big([-\pi,\pi]\big)$;
        \item $\{D_n(a)\}_n \sim_\glt \kappa(x,\theta) = a(x)$ for any $a:[0,1] \to\C$ continuous a.e.;
        \item $\{Z_n\}_n \sim_\glt \kappa(x,\theta) = 0$ if and only if $\{Z_n\}_n \sim_\sigma 0$, i.e., $\{Z_n\}_n$ is zero-distributed.
    \end{itemize}
\end{gltax}

\begin{gltax} \label{glt:algebra}
    If $\{A_n\}_n \sim_\glt \kappa$ and $\{B_n\}_n \sim_\glt \varsigma$, then
    \begin{itemize}
        \item $\{A_n^*\}_n \sim_\glt \bar{\kappa}$, where $A_n^*$ denotes the conjugate transpose of $A_n$ and $\bar{\kappa}$ the complex conjugate of $\kappa$;
        \item $\{\alpha A_n + \beta B_n\}_n \sim_\glt \alpha\kappa +\beta\varsigma$ for any $\alpha,\beta\in\C$;
        \item $\{A_n B_n\}_n \sim_\glt \kappa\varsigma$;
        \item $\{A_n^\dagger\}_n \sim_\glt\kappa^{-1}$ for any $\kappa$ invertible a.e., where $A_n^\dagger$ denotes the Moore-Penrose pseudoinverse of $A_n$.
    \end{itemize}
\end{gltax}

\begin{gltax} \label{glt:kron}
    Let $\{A_n\}_n \sim_\glt \kappa$ and $\{B_m\}_m \sim_\glt \xi$ with
    \begin{align*}
        &\kappa(x_1,\theta_1) : [0,1]\times[-\pi,\pi] \to\C,
        \\
        &\xi(x_2,\theta_2) : [0,1]\times[-\pi,\pi]\to\C.
    \end{align*}
    
    Then, setting $N = N(n,m) := nm$, it holds $\{A_n \otimes B_m\}_N \sim_\glt \kappa\otimes\xi$, where $(\kappa\otimes\xi): [0,1]^2\times[-\pi,\pi]^2\to\C$ is given by
    \begin{equation*}
        (\kappa\otimes\xi)(x_1,x_2,\theta_1,\theta_2) := \kappa(x_1,\theta_1)\xi(x_2,\theta_2),
    \end{equation*}
    i.e. we obtain a twolevel GLT matrix sequence ($d=2$). Notice that the axioms in the $d$-level setting are verbatim the same with $[0,1]$ and $[-\pi,\pi]$ replaced by $[0,1]^d$ and $[-\pi,\pi]^d$, respectively.
\end{gltax}

\section{Discretization and matrix structures} \label{sec:discre}

We now present the discretization of the regularized problem \eqref{eq:quasi-problem}, using a finite difference scheme \cite{Sun2020}. Then we describe with care the global structure of the arising coefficient matrix.


Let $m$ and $n$ be positive integers and denote the corresponding spatial and temporal meshes as
\begin{alignat}{3}
    x_i &= i\Delta x, && \qquad\Delta x = \frac{1}{m+1}, && \qquad i=0,1,\ldots,m, \label{eq:space-mesh} \\
    t_j &= j\Delta t, && \qquad\Delta t=\frac{T}{n}, && \qquad j=0,1,\ldots,n. \label{eq:time-mesh}
\end{alignat}

To deal with the tempered fractional derivative, we use the $L1$ formula:
\begin{equation*}
    {}^T\!D_{L1}^{\xi,\rho} u(x,t_j) = \frac{(\Delta t)^{-\xi}}{\Gamma(2-\xi)} \left[ b_0 u(x,t_j) - \sum_{l=1}^{j-1} \left( b_{j-l-1} - b_{j-l} \right) e^{\rho(t_l-t_j)} u(x,t_l) - b_{j-1} e^{\rho(t_0-t_j)} u(x,t_0) \right],
\end{equation*}
in which
\begin{equation*}
    b_l = (l+1)^{1-\xi}-l^{1-\xi},
    \qquad l = 0,1,\ldots
\end{equation*}
Denoting
\begin{equation*}
    \alpha_n := (\Delta t)^\xi\Gamma(2-\xi)
    = \bigg(\frac{T}{n}\bigg)^\xi\Gamma(2-\xi)
\end{equation*}
and
\begin{equation} 
    \begin{split} \label{eq:time-coeff}
        \gamma_0 &:= b_0, \\
        \gamma_l &:= b_l - b_{l-1} = (l+1)^{1-\xi} - 2l^{1-\xi} + (l-1)^{1-\xi}, \qquad l=1,\ldots,n-1,
    \end{split}
\end{equation}
and using $t_j = j\Delta t$, the $L1$ approximation of the tempered fractional derivative is expressed as
\begin{equation} \label{eq:tem-discre}
    {}^T\!D_{L1}^{\xi,\rho} u(x,t_j) = \frac{1}{\alpha_n} \left[ \gamma_0 u(x,t_j) + \sum_{l=1}^{j-1} \gamma_{j-l} e^{-(j-l)\rho\Delta t} u(x,t_l) - b_{j-1} e^{-j \rho \Delta t} u(x,t_0) \right].
\end{equation}
The following error estimate holds. Assuming that $u(\cdot,t)\in C^2(0,T)$, for $t_j\in (0,T)$ we have
\begin{equation*}
    \abs{{}^T\!D_t^{\xi,\rho} u(x,t)|_{t=t_j}-{}^T\!D_{L1}^{\xi,\rho}u(x,t_j)} \leq \chi_\xi
    \left[ \rho^2 \max_{t_0\leq t\leq t_j} \abs{u(x,t)} + 2\rho \max_{t_0\leq t\leq t_j} \abs{u_t(x,t)} + \Delta t^{2-\xi} \max_{t_0\leq t\leq t_j} \abs{u_{tt}(x,t)} \right],
\end{equation*}
where $u_t$ and $u_{tt}$ denote the first and second partial derivative of $u$ with respect to $t$, and
\begin{equation*}
    \chi_\xi :=  \frac{1}{2\Gamma(1-\xi)} \left[ \frac{1}{4} + \frac{\xi}{(1-\xi)(2-\xi)} \right].
\end{equation*}

To discretize the Caputo fractional derivative, we exploit the spatial counterpart of the $L1$ formula:
\begin{equation} \label{eq:Caputo-discre}
    {}^C\!D_{L1}^\eta u(x_i,t) = \frac{1}{\beta_m} \left[ \delta_0 u(x_i,t) + \sum_{k=1}^{i-1} \delta_{i-k} u(x_k,t) - d_{i-1} u(x_0,t) \right],
\end{equation}
where
\begin{equation*}
    \beta_m := (\Delta x)^\eta \Gamma(2-\eta) = \bigg(\frac{1}{m+1}\bigg)^\eta \Gamma(2-\eta)
\end{equation*} and
\begin{equation} \label{eq:space-coeff}
    \begin{split}
        d_k &:= (k+1)^{1-\eta}-k^{1-\eta},
        \qquad k = 0,1,\ldots \\
        \delta_0 &:= d_0, \\
        \delta_k &:= d_k - d_{k-1} = (k+1)^{1-\eta} - 2k^{1-\eta} + (k-1)^{1-\eta}, \qquad k=1,\ldots,m-1.
    \end{split}
\end{equation}
The corresponding error estimate takes the form
\begin{equation*}
    \abs{{}^C\!D_x^\eta u(x,t)|_{x=x_i}-{}^C\!D_{L1}^\eta u(x_i,t)} \leq
    \frac{1}{2\Gamma(1-\eta)} \left[ \frac{1}{4} + \frac{\eta}{(1-\eta)(2-\eta)} \right] \Delta x^{2-\eta} \max_{0\leq x \leq x_i} \abs{u_{xx}(x,t)},
\end{equation*}
with $u(x,\cdot)\in C^2(0,1)$ and $x_i\in (0,1)$.

Substituting the fractional derivatives in problem \eqref{eq:quasi-problem} with their respective L1 approximations \eqref{eq:tem-discre} and \eqref{eq:Caputo-discre}, and omitting the truncation errors, the finite difference scheme is constructed as follows: in $(x_i,y_j)$ for $i=1,\ldots,m$ and $j=1,\ldots,n$ we have
\begin{equation*}
    \begin{dcases}
        \frac{1}{\alpha_n} \left[ \gamma_0 u_i^{(j)} + \sum_{l=1}^{j-1} \gamma_{j-l} e^{-(j-l)\rho\Delta t} u_i^{(l)} - b_{j-1} e^{-j\rho\Delta t} u_i^{(0)} \right]
        + \frac{a_i}{\beta_m}
        \left[ \delta_0 u_i^{(j)} + \sum_{k=1}^{i-1} \delta_{i-k} u_k^{(j)} - d_{i-1} u_0^{(j)} \right]
        = q^{(j)} f_i,
        \\
        u_0^{(j)} = u_{m+1}^{(j)} = 0,
        \\
        u_i^{(0)} = \varphi_i,
        \\
        u_i^{(n)} = \phi_{\varepsilon,i} - \lambda \frac{a_i}{\beta_m}
        \left[ \delta_0 f_i + \sum_{k=1}^{i-1} \delta_{i-k} f_k - d_{i-1} f_0 \right],
    \end{dcases}
\end{equation*}
where
\begin{alignat*}{3}
    u_i^{(j)} &:= u(x_i,t_j), & \qquad
    a_i &:= a(x_i), & \qquad
    f_i &:= f_{\varepsilon,\lambda}(x_i), 
    \\
    \varphi_i &:= \varphi(x_i), & \qquad
    \phi_{\varepsilon,i} &:= \phi_\varepsilon(x_i), & \qquad
    q^{(j)} &:= q(t_j).
\end{alignat*}
By defining the corresponding vectors
\begin{alignat*}{3}
    \mathbf{u}^{(j)} &:= \begin{bmatrix} u_i^{(j)} \end{bmatrix}_{i=1}^m, & \qquad
    \mathbf{a} &:= \begin{bmatrix} a_i \end{bmatrix}_{i=1}^m, & \qquad
    \mathbf{f} &:= \begin{bmatrix} f_i \end{bmatrix}_{i=1}^m,
    \\
    \boldsymbol{\varphi} &:= \begin{bmatrix} \varphi_i \end{bmatrix}_{i=1}^m, & \qquad
    \boldsymbol{\phi}_\varepsilon &:= \begin{bmatrix} \phi_{\varepsilon,i} \end{bmatrix}_{i=1}^m, 
\end{alignat*}
we can write the numerical scheme in matrix form as
\begin{equation*}
    \begin{dcases}
        \gamma_0 \mathbf{u}^{(j)} + \sum_{l=1}^{j-1} \gamma_{j-l} e^{-(j-l)\rho\Delta t} \mathbf{u}^{(l)} 
        + \frac{\alpha_n}{\beta_m} D_m(a) B_m \mathbf{u}^{(j)}
        - \alpha_n q^{(j)} \mathbf{f}
        = b_{j-1} e^{-j\rho\Delta t} \boldsymbol{\varphi},
        & \quad j=1,\ldots,n, \\
        \mathbf{u}^{(n)} + \frac{\lambda}{\beta_m} D_m(a) B_m \mathbf{f} = \boldsymbol{\phi}_\varepsilon,
    \end{dcases}
\end{equation*}
where $D_m(a)$ is the diagonal sampling matrix of size $m$ associated to $a(x)$, while $B_m$ is the following lower triangular Toeplitz matrix
\begin{equation*}
    B_m := \begin{bmatrix}
        \delta_0 & 0 & \cdots & \cdots & 0 \\
        \delta_1 & \delta_0 & 0 & & \vdots \\
        \delta_2 & \delta_1 & \delta_0 & \ddots & \vdots \\
        \vdots & \ddots & \ddots & \ddots& 0  \\
        \delta_{m-1} & \cdots & \delta_2 & \delta_1 & \delta_0
    \end{bmatrix}.
\end{equation*}
By gathering the equations for all the time levels and the regularization equation, we obtain the all-at-once linear system
\begin{equation} \label{eq:lin-syst}
    A_N\mathbf{v}_\lambda = \mathbf{z}_\lambda,
\end{equation}
where $N = N(n,m) := (n+1)m$ and
\begin{gather*}
    \renewcommand{\arraystretch}{1.5}
    A_N := \left[ \begin{array}{cccc|c}
        \gamma_0 I_m + \frac{\alpha_n}{\beta_m} G_m & \mathrm{O}_m & \cdots & \mathrm{O}_m & -\alpha_n q^{(1)} I_m \\
        \gamma_1 e^{-\rho \Delta t} I_m & \gamma_0 I_m + \frac{\alpha_n}{\beta_m} G_m & \ddots & \vdots & -\alpha_n q^{(2)} I_m \\
        \vdots & \ddots & \ddots & \mathrm{O}_m & \vdots \\
        \gamma_{n-1} e^{-(n-1)\rho\Delta t} I_m & \cdots & \gamma_1 e^{-\rho \Delta t} I_m & \gamma_0 I_m + \frac{\alpha_n}{\beta_m} G_m & -\alpha_n q^{(n)} I_m \\
        \hline
        \mathrm{O}_m & \cdots & \mathrm{O}_m & I_m & \frac{\lambda}{\beta_m} G_m
    \end{array} \right]
    \in \C^{N\times N},
    \\
    \renewcommand{\arraystretch}{1.5}
    \mathbf{v}_{\lambda} :=
        \left[ \begin{array}{c}
        \mathbf{u}^{(1)} \\ \mathbf{u}^{(2)} \\ \vdots \\ \mathbf{u}^{(n)} \\ \hline \mathbf{f}
    \end{array} \right]
    \in\C^N,
    \qquad
    \mathbf{z}_{\lambda} :=
        \left[ \begin{array}{c}
            b_0 e^{-\rho\Delta t} \boldsymbol{\varphi} \\
            b_1 e^{-2\rho\Delta t} \boldsymbol{\varphi} \\
            \vdots \\
            b_{n-1} e^{-n\rho\Delta t} \boldsymbol{\varphi} \\
            \hline \boldsymbol{\phi}_\varepsilon
        \end{array} \right]
        \in\C^N,
\end{gather*}
in which $I_m$ is the $m\times m$ identity matrix, $\mathrm{O}_m$ is the $m\times m$ zero matrix, and
\begin{equation} \label{eq:Gm-def}
    G_m := D_m(a) B_m.
\end{equation}
Moreover, setting
\begin{equation*}
    U_n := \begin{bmatrix}
        \gamma_0 & 0 & \cdots & 0 \\
        \gamma_1 e^{-\rho\Delta t} & \gamma_0 & \ddots & \vdots \\
        \vdots & \ddots & \ddots & 0 \\
        \gamma_{n-1} e^{-\rho(n-1)\Delta t} & \cdots & \gamma_1 e^{-\rho\Delta t} & \gamma_0
    \end{bmatrix},
\end{equation*}
the coefficient matrix can be expressed into the compact form
\begin{equation} \label{eq:coeff-mat}
    \renewcommand{\arraystretch}{1.5}
    A_N = \left[ \begin{array}{c|c}
    	U_n \otimes I_m + I_n\otimes \frac{\alpha_n}{\beta_m} G_m & -\alpha_n \mathbf{q}\otimes I_m \\
        \hline
    	\mathbf{e}_n^T \otimes I_m & \frac{\lambda}{\beta_m} G_m
    \end{array} \right],
\end{equation}
in which
\begin{equation*}
    \mathbf{e}_n := \begin{bmatrix}
        0 \\ \vdots \\ 0 \\ 1
    \end{bmatrix} \in\R^n,
    \qquad
    \mathbf{q} := \begin{bmatrix}
        q^{(1)} \\ q^{(2)} \\ \vdots \\ q^{(n)}
    \end{bmatrix} \in\R^n.
\end{equation*}

\section{GLT analysis} \label{sec:glt-analysis}

This section is devoted to the spectral analysis of the coefficient matrix sequence via the GLT theory, in other to gather information on the asymptotic behavior of the matrices. We start with the sequence associated with the space block $G_m$, then proceed to the time block $U_n$, and conclude with the overall coefficient matrix sequence.

\begin{theorem} \label{thm:GLT-Bm}
    There exists a continuous $2\pi$-periodic function $g_\eta$ such that
    \begin{equation*}
        \{B_m\}_m \sim_{\glt,\sigma} g_\eta(\theta),
        \qquad \theta\in [-\pi,\pi].
    \end{equation*}
    Moreover,
    \begin{equation*}
        \{B_m\}_m \sim_\lambda \delta_0.
    \end{equation*}
\end{theorem}
\begin{proof}
    To demonstrate that $\{B_m\}_m$ is a GLT sequence, we claim that $B_m$ is a Toeplitz matrix generated by a measurable function, i.e., there exists $g_\eta$ such that $B_m = T_m(g_\eta)$ for all $m\in\N$. The  idea is simple: we prove that the series
    \begin{equation*}
        g_\eta(\theta) := \sum_{k=0}^{\infty} \delta_k e^{\iu k\theta},
    \end{equation*}
    where the coefficients $\delta_k$ are given by \eqref{eq:space-coeff}, belongs to the Wiener class, so that $g_\eta$ is well defined, continuous and $2\pi$-periodic, and the corresponding Toeplitz matrix sequence is precisely $\{B_m\}_m$ by definition.
    
    Let us consider the auxiliary function $\psi_\eta(x) := x^{1-\eta}$. From \eqref{eq:space-mesh}, for any index $k$ we have $k = (m+1)x_k$. Hence,
    \begin{align*}
        \delta_k &= (k+1)^{1-\eta} - 2k^{1-\eta} + (k-1)^{1-\eta}
        \\
        &= (m+1)^{1-\eta} \big[ x_{k+1}^{1-\eta} - 2 x_k^{1-\eta} + x_{k-1}^{1-\eta} \big]
        \\
        &= (m+1)^{1-\eta} \big[ \psi_\eta(x_{k+1}) - 2\psi_\eta(x_k) + \psi_\eta(x_{k-1}) \big]
        \\
        &= (m+1)^{1-\eta} \left[ \frac{ \psi_\eta(x_{k+1}) - 2\psi_\eta(x_k) + \psi_\eta(x_{k-1}) }{\Delta x^2} \right] \Delta x^2
        \\
        &= \frac{1}{(m+1)^{1+\eta}} \big[ \psi_\eta''(x_k) + O(\Delta x^2) \big] \qquad \text{as } \Delta x\to 0,
    \end{align*}
    where we used a standard second order central finite difference scheme and the fact that $\Delta x = \frac{1}{m+1}$. From
    \begin{equation*}
        \psi_\eta''(x) = \eta(\eta-1) \frac{1}{x^{1+\eta}},
    \end{equation*}
    we have
    \begin{align*}
        \delta_k
        &= \frac{\eta(\eta-1)}{(m+1)^{1+\eta}} \cdot  \frac{1}{x_k^{1+\eta}} + \Delta x^{1+\eta} O(\Delta x^2) \\
        &= \frac{\eta(\eta-1)}{(m+1)^{1+\eta}} \cdot \frac{(m+1)^{1+\eta}}{k^{1+\eta}} + O(\Delta x^{3+\eta}) \\
        &= \eta(\eta-1) \frac{1}{k^{1+\eta}} + O(\Delta x^{3+\eta}).
    \end{align*}
    Since $\eta\in (0,1)$, we deduce that the series
    \begin{equation*}
        \sum_{k=0}^{\infty} \abs{\delta_k}
    \end{equation*}
    is convergent, and consequently $g_\eta$ belongs to the Wiener class. Therefore, $B_m = T_m(g_\eta)$, so that GLT \ref{glt:symbols} and GLT \ref{glt:distrib} imply
    \begin{equation*}
        \{B_m\}_m \sim_{\glt,\sigma} g_\eta(\theta).
    \end{equation*}
    Finally, all the eigenvalues of $B_m$ are trivially equal to $\delta_0$ because of the triangular structure, so that 
    \begin{equation*}
        \{B_m\}_m \sim_\lambda \delta_0
    \end{equation*}
    follows immediately from the definition of distribution.
\end{proof}

\begin{corollary}
    If $a:[0,1]^2\to\R$ is continuous a.e, then
    \begin{equation*}
        \{G_m\}_m  \sim_{\glt,\sigma} a(x) g_\eta(\theta),
        \qquad x\in [0,1], \;\theta\in [-\pi,\pi].
    \end{equation*}
    with $g_\eta$ as in Theorem \ref{thm:GLT-Bm}. Moreover,
    \begin{equation*}
        \{G_m\}_m  \sim_\lambda \delta_0 a(x),
        \qquad x\in [0,1].
    \end{equation*}
\end{corollary}
\begin{proof}
    Recalling \eqref{eq:Gm-def}, by GLT \ref{glt:symbols} we have
    \begin{equation*}
        \{D_m(a)\}_m \sim_\glt a(x), \qquad x\in [0,1].
    \end{equation*}
    Then, by Theorem \ref{thm:GLT-Bm} and the $\ast$-algebra properties detailed in GLT \ref{glt:algebra}, we deduce
    \begin{equation*}
        \{G_m\}_m \sim_\glt a(x)g_\eta(\theta),
        \qquad x\in [0,1], \;\theta\in [-\pi,\pi].
    \end{equation*}
    By GLT \ref{glt:distrib}, we immediately have
    \begin{equation*}
        \{G_m\}_m \sim_\sigma a(x)g_\eta(\theta),
        \qquad x\in [0,1], \;\theta\in [-\pi,\pi].
    \end{equation*}
    Finally, we notice that $G_m$ is lower triangular and the main diagonal is given by $\begin{bmatrix} \delta_0 a(x_i) \end{bmatrix}_{i=1}^m$. In conclusion
    \begin{equation*}
        \{G_m\}_m \sim_\lambda \delta_0 a(x),
        \qquad x\in [0,1].
    \end{equation*}
\end{proof}

We proceed to the analysis of $\{U_n\}_n$, relying on the same technique employed in the proof of Theorem \ref{thm:GLT-Bm}.

\begin{theorem} \label{thm:GLT-Un}
    There exists a continuous $2\pi$-periodic function $h_\xi$ such that
    \begin{equation*}
        \{U_n\}_n \sim_{\glt,\sigma} h_\xi(\theta),
        \qquad \theta\in [-\pi,\pi].
    \end{equation*}
    Moreover,
    \begin{equation*}
        \{U_n\}_n \sim_\lambda \gamma_0.
    \end{equation*}
\end{theorem}

\begin{proof}
    We follow the proof of Theorem \ref{thm:GLT-Bm}, showing that there exists $h_\xi$ such that $U_n = T_n(h_\xi)$ for all $n\in\N$. We prove that the series    \begin{equation*}
        h_\xi(\theta) := \sum_{k=0}^{\infty} \gamma_k e^{-\rho k \Delta t} e^{\iu k\theta},
    \end{equation*}
    with $\gamma_k$ defined by \eqref{eq:time-coeff}, belongs to the Wiener class. Considering $\psi_\xi(t) := t^{1-\xi}$ and using \eqref{eq:time-mesh}, we have
    \begin{align*}
        \gamma_k &= (k+1)^{1-\xi} - 2k^{1-\xi} + (k-1)^{1-\xi}
        \\
        &= \left( \frac{n}{T} \right)^{1-\xi} \left[ t_{k+1}^{1-\xi} - 2 t_k^{1-\xi} + t_{k-1}^{1-\xi} \right]
        \\
        &= \left( \frac{n}{T} \right)^{1-\xi} \big[ \psi_\xi(t_{k+1}) - 2\psi_\xi(t_k) + \psi_\xi(t_{k-1}) \big]
        \\
        &= \left( \frac{n}{T} \right)^{1-\xi} \left[ \frac{ \psi_\xi(t_{k+1}) - 2\psi_\xi(t_k) + \psi_\xi(t_{k-1}) }{\Delta t^2} \right] \Delta t^2
        \\
        &= \left(\frac{T}{n}\right)^{1+\xi} \big[ \psi_\xi''(t_k) + O(\Delta t^2) \big] \qquad \text{as } k\to\infty,
    \end{align*}
    using $\Delta t = \frac{T}{n}$ and a standard central finite difference scheme. From
    \begin{equation*}
        \psi_\xi''(x) = \xi(\xi-1) \frac{1}{t^{1+\xi}},
    \end{equation*}
    we have
    \begin{align*}
        \gamma_k
        &= \left(\frac{T}{n}\right)^{1+\xi} \xi(\xi-1) \frac{1}{t_k^{1+\xi}} + \Delta t^{1+\xi} O(\Delta t^2) \\
        &= \left(\frac{T}{n}\right)^{1+\xi} \frac{\xi(\xi-1)}{k^{1+\xi}} \left(\frac{n}{T}\right)^{1+\xi} + O(\Delta t^{3+\xi}) \\
        &= \frac{\xi(\xi-1)}{k^{1+\xi}} + O(\Delta t^{3+\xi}).
    \end{align*}
    Since $\xi\in (0,1)$ and $\rho>0$, the series
    \begin{equation*}
        \sum_{k=0}^{\infty} \abs{\gamma_k e^{-\rho k \Delta t}}
    \end{equation*}
    is convergent, and $h_\xi$ belongs to the Wiener class by definition. Hence, $U_n = T_n(h_\xi)$ and by GLT \ref{glt:symbols} and GLT \ref{glt:distrib} we conclude
    \begin{equation*}
        \{U_n\}_n \sim_{\glt,\sigma} h_\xi(\theta),
    \end{equation*}
    while the eigenvalue distribution is obvious from the triangular structure.
\end{proof}

We are now ready to compute the GLT symbol of the sequence associated with the full coefficient matrix. To provide meaningful results, we introduce a mild assumption regarding the rate at which the discretization parameters tend to zero with respect to one another.

\begin{theorem} \label{thm:coeff-mat-sym}
    With $A_N$ defined in \eqref{eq:coeff-mat}, assume that
    \begin{equation*}
        \frac{\alpha_n}{\beta_m} \to \nu, \qquad \text{as } n,m\to\infty,
    \end{equation*}
    with $\nu\neq 0$. Then, it holds
    \begin{equation*}
        \{A_N\}_N \sim_{\glt,\sigma} h_\xi(\theta_1) + \nu a(x)g_\eta(\theta_2), \qquad (\theta_1,\theta_2)\in [-\pi,\pi]^2, \; x\in [0,1],
    \end{equation*}
    where $h_\xi$ and $g_\eta$ are given in Theorems \ref{thm:GLT-Bm}-\ref{thm:GLT-Un}.
\end{theorem}

\begin{proof}
We begin by decomposing \eqref{eq:coeff-mat} as
\begin{equation*}
    \renewcommand{\arraystretch}{1.5}
    A_N = \underbrace{ \left[ \begin{array}{cc}
    	U_n \otimes I_m & \mathrm{O}_{nm\times m} \\
    	\mathrm{O}_{m\times nm} & \mathrm{O}_m
    \end{array} \right] }_{=:\, C_N^{(1)}}
    + \underbrace{ \left[ \begin{array}{cc}
    	I_n\otimes \frac{\alpha_n}{\beta_m} G_m & \mathrm{O}_{nm\times m} \\
    	\mathrm{O}_{m\times nm} & \frac{\lambda}{\beta_m} G_m
    \end{array} \right] }_{=:\, C_N^{(2)}}
    + \underbrace{ \left[ \begin{array}{cc}
    	\mathrm{O}_{nm} & -\alpha_n\mathbf{q}\otimes I_m \\
    	\mathbf{e}_n^T \otimes I_m & \mathrm{O}_m
    \end{array} \right] }_{=:\, C_N^{(3)}},
\end{equation*}
so we can focus on each component individually. For the first one, we have
\begin{equation*}
    \renewcommand{\arraystretch}{1.3}
    C_N^{(1)}
    = U_{n+1} \otimes I_m
    - \underbrace{\left[ \begin{array}{c|c}
    	\mathrm{O}_{nm} & \mathrm{O}_{nm\times m} \\
    	\hline
    	\begin{matrix}
    	    \gamma_n e^{-\rho n\Delta t} I_m & \cdots & \gamma_1 e^{-\rho\Delta t} I_m
    	\end{matrix} & \gamma_0 I_m
    \end{array} \right]}_{=: R^{(1)}_N},
\end{equation*}
where $\rk R^{(1)}_N \leq m = o(N)$ as $n,m\to\infty$. By Lemma \ref{lemma:zero-distr}, $\{R^{(1)}_N\}_N$ is zero-distributed, hence by GLT \ref{glt:symbols}
\begin{equation*}
    \{R^{(1)}_N\}_N\sim_\glt 0.
\end{equation*}
Moreover, we have $U_{n+1} \otimes I_m = T_{n+1}(h_\xi) \otimes T_m(1)$, therefore applying GLT \ref{glt:symbols}--\ref{glt:kron} leads to
\begin{equation} \label{eq:S1}
    \big\{C_N^{(1)}\big\}_N \sim_\glt h_\xi(\theta_1).
\end{equation}

Concerning $C_N^{(2)}$, we notice that
\begin{equation*}
    C_N^{(2)} = \left[ \begin{array}{cc}
    	I_n\otimes \frac{\alpha_n}{\beta_m} G_m & \mathrm{O}_{nm\times m} \\
    	\mathrm{O}_{m\times nm} & \frac{\alpha_n}{\beta_m} G_m + \frac{\lambda-\alpha_n}{\beta_m} G_m
    \end{array} \right]
    = I_{n+1} \otimes \frac{\alpha_n}{\beta_m} G_m +
    \underbrace{\left[ \begin{array}{cc}
    	\mathrm{O}_{nm} & \mathrm{O}_{nm\times m} \\
    	\mathrm{O}_{m\times nm} & \frac{\lambda-\alpha_n}{\beta_m} G_m
    \end{array} \right]}_{=: R_N^{(2)}},
\end{equation*}
in which once again $\rk R_N^{(2)} = m = o(N)$ as $n,m\to\infty$, so that
\begin{equation*}
    \{R_N^{(2)}\}_N\sim_\glt 0.
\end{equation*}
On the other hand, with the assumption that $\frac{\alpha_n}{\beta_m}\to \nu$, by GLT \ref{glt:symbols} and GLT \ref{glt:kron} we have
\begin{equation*}
    \left\{I_{n+1} \otimes \frac{\alpha_n}{\beta_m} G_m\right\}_N = \left\{ T_{n+1}(1) \otimes \frac{\alpha_n}{\beta_m} D_m(a)T_m(g_\eta) \right\}_N \sim_\glt \nu a(x) g_\eta(\theta_2).
\end{equation*}
We conclude by GLT \ref{glt:algebra} that
\begin{equation} \label{eq:S2}
    \big\{C_N^{(2)}\big\}_N \sim_\glt \nu a(x) g_\eta(\theta_2).
\end{equation}

Finally, we have
\begin{equation*}
    C_N^{(3)} = \left[ \begin{array}{cc}
    	\mathrm{O}_{nm} & -\alpha_n\mathbf{q}\otimes I_m \\
    	\mathrm{O}_{m\times nm} & \mathrm{O}_m
    \end{array} \right]
    + \left[ \begin{array}{cc}
    	\mathrm{O}_{nm} & \mathrm{O}_{nm\times m} \\
    	\mathbf{e}_n^T \otimes I_m & \mathrm{O}_m
    \end{array} \right],
\end{equation*}
where each of these components has rank at most $m=o(N)$. Therefore, 
\begin{equation} \label{eq:S3}
    \big\{C_N^{(3)}\big\}_N \sim_\glt 0.
\end{equation}

By summing the contributions in \eqref{eq:S1}, \eqref{eq:S2} and \eqref{eq:S3}, we find the GLT symbol of the overall coefficient matrix sequence $\{A_N\}_N$ via GLT \ref{glt:algebra}, and then the singular value distribution via GLT \ref{glt:distrib}.

\end{proof}

\section{GLT preconditioning} \label{sec:precond}

In the present section, we propose preconditioners $P_N$ for the linear system \eqref{eq:lin-syst}, with the aim of clustering at 1 the eigenvalues of $\{ P_N^{-1} A_N \}_N $. Let us define
\begin{equation} \label{eq:precond}
    \renewcommand{\arraystretch}{1.5}
    P_N := \left[ \begin{array}{c|c}
    	U_n \otimes I_m + I_n\otimes \frac{\alpha_n}{\beta_m} G_m & \mathrm{O}_{mn \times m} \\
        \hline
    	\mathbf{e}_n^T \otimes I_m & \frac{\lambda}{\beta_m} G_m
    \end{array} \right].
\end{equation}
Note that both $U_n$ and $G_m$ are invertible, in consequence of their lower triangular structure with non-zero diagonal elements, and the same is true for $U_n \otimes I_m + I_n\otimes \frac{\alpha_n}{\beta_m} G_m$, which is also lower triangular. Because of the block triangular structure, the eigenvalues of $P_N$ are the eigenvalues of the diagonal blocks, therefore $P_N$ is invertible.

The following theorem proves that $\{P_N\}_N$ has the same GLT symbol, and hence the same singular value distribution, as the coefficient matrix sequence. 

\begin{theorem} \label{thm:precond-sym}
    With $P_N$ defined in \eqref{eq:precond}, assume that
    \begin{equation*}
        \frac{\alpha_n}{\beta_m} \to \nu, \qquad \text{as } n,m\to\infty,
    \end{equation*}
    with $\nu\neq 0$. Then, it holds
    \begin{equation*}
        \{P_N\}_N \sim_{\glt,\sigma} h_\xi(\theta_1) + \nu a(x)g_\eta(\theta_2), \qquad (\theta_1,\theta_2)\in [-\pi,\pi]^2, \; x\in [0,1],
    \end{equation*}
    where $h_\xi$ and $g_\eta$ are given in Theorems \ref{thm:GLT-Bm}-\ref{thm:GLT-Un}.
\end{theorem}

\begin{proof}
    Immediate by following the same reasoning as in the proof of Theorem \ref{thm:coeff-mat-sym}.
\end{proof}

The GLT symbol and singular value distribution of the preconditioned matrix sequence can be derived as a corollary of Theorems \ref{thm:coeff-mat-sym} and \ref{thm:precond-sym}, simply by exploiting the algebra properties in GLT \ref{glt:algebra}. However, it can also be computed more in general avoiding the assumption on the discretization parameters, as we show in the following theorem. Moreover, we can derive the distribution of the eigenvalues.

\begin{theorem} 
    With $A_N$ as in \eqref{eq:coeff-mat} and $P_N$ as in \eqref{eq:precond}, it holds
    \begin{equation*}
        \{P_N^{-1}A_N\}_N \sim_{\glt,\sigma} 1.
    \end{equation*}
    Moreover,
    \begin{equation*}
        \{P_N^{-1}A_N\}_N \sim_\lambda 1.
    \end{equation*}
\end{theorem}

\begin{proof}
    Setting $R_N := A_N - P_N$, we have
    \begin{equation*}
         P_N^{-1} A_N = I_N + P_N^{-1}R_N.
    \end{equation*}
    From the fact that $\rk R_N\leq m$, it follows $\rk(P_N^{-1}R_N) \leq m = o(N)$, therefore
    \begin{equation*}
        \{P_N^{-1}R_N\}_N \sim_\glt 0
    \end{equation*}
     by Lemma \ref{lemma:zero-distr}. Hence, by GLT \ref{glt:algebra} and the fact that $\{I_N\}_N \sim_\glt 1$,
    \begin{equation*}
        \{P_N^{-1} A_N\}_N \sim_\glt 1.
    \end{equation*}
    The singular value distribution follows immediately from GLT \ref{glt:distrib}. Concerning the eigenvalue distribution, it is obvious that
    \begin{equation*}
        \lambda_j (P_N^{ -1 } A_N ) = 1 + \lambda_j (P_N^{-1} R_N),
        \qquad j=1,\ldots,N,
    \end{equation*}
    therefore
    \begin{equation*}
        \# \left\{ j=1,\ldots,N  \;\colon\; \lambda_j (P_N^{-1} A_N) \neq 1 \right\} \leq m = o(N).
    \end{equation*}
    In other words, the eigenvalues of $\{ P_N^{-1} A_N \}_N$ are weakly clustered at 1, in the sense of Definition \ref{clustering}. Recalling Remark \ref{remark:distrib-cluster}, this is equivalent to the fact that the sequence is spectrally distributed as 1.
\end{proof}

\begin{remark}
    From an implementation viewpoint, we can define a custom function that inverts $P_N$ through a forward substitution method, thanks to the block triangular structure of the matrix and the triangular structure of the blocks themselves. The computational cost amounts to $O(m^2n^2)$.
\end{remark}

\begin{remark}
    In the special case where $a(x)=a$ is constant, so that $D_m(a)=aI_m$ and $G_m = aB_m$, a more efficient preconditioner can be constructed by approximating each Toeplitz matrix appearing in $P_N$ with the corresponding Strang preconditioner, as follows:
    \begin{equation*}
        \renewcommand{\arraystretch}{1.5}
        S_N := \left[ \begin{array}{c|c}
        	S(U_n) \otimes I_m + I_n\otimes \frac{\alpha_n}{\beta_m} a S(B_m) & \mathrm{O}_{mn \times m} \\
            \hline
        	\mathbf{e}_n^T \otimes I_m & \frac{\lambda}{\beta_m} a S(B_m)
        \end{array} \right],
    \end{equation*}
    where $S(U_n)$, $S(B_m)$ are the standard Strang preconditioners of $U_n$ and $B_m$, respectively. Since circulant matrices are uniformly diagonalizable through the Fourier matrix $F_n$, we can write
    \begin{equation*}
        S(U_n) = F_n \Lambda_n^{(1)} F_n^H,
        \qquad
        S(B_m) = F_m \Lambda_m^{(2)} F_m^H,
    \end{equation*}
    where $\Lambda_n^{(i)}$, $i=1,2$, are the diagonal matrices containing the eigenvalues and can be computed starting from the coefficients of $U_n$, $B_m$ via the Fast Fourier Transform (FFT). Then, $S_N$ becomes
    \begin{equation*}
        \renewcommand{\arraystretch}{1.5}
        S_N := \left[ \begin{array}{c|c}
        	(F_n\otimes F_m)
            \big[\Lambda_n^{(1)}\otimes I_m + I_n\otimes \frac{\alpha_n}{\beta_m} a \Lambda_m^{(2)} \big](F_n\otimes F_m)^H & \mathrm{O}_{mn \times m} \\
            \hline
        	\mathbf{e}_n^T \otimes I_m & \frac{\lambda}{\beta_m} a F_m \Lambda_m^{(2)} F_m^H
        \end{array} \right].
    \end{equation*}
    With this approach, the computational cost of inverting $S_N$ drops to $O\big(mn\log(mn)\big)$, since all the necessary computations can be performed via the FFT. The same reasoning can be carried out using any suitable preconditioner for Toeplitz matrices, drawn from a matrix algebra which is uniformly diagonalizable through a fast transform.
\end{remark}

\section{Numerical results}\label{sec:num}

We conclude by assessing the efficiency of our preconditioning strategy in terms of iterations required by the GMRES.

Let us consider problem \eqref{eq:quasi-problem} in which we set
\begin{gather*}
    T = 1, \qquad a(x) = x, \qquad
    q(t) = t^2, \qquad \varphi(x) = 0.
\end{gather*}

To construct the noisy final data, we select the source function
\begin{equation*}
    f(x) = x\sin(\pi x)
\end{equation*}
and solve the linear system arising from the corresponding direct problem
\begin{equation*}
    \begin{dcases}
        {}^T\!D_t^{\xi,\rho} u(x,t) + a(x) {}^C\!D_x^{\eta} u(x,t) = q(t)f(x),
        & x\in (0,1), \; t\in (0,T),
        \\
        u(0,t) = u(1,t) = 0,
        & t\in (0,T),
        \\
        u(x,0) = \varphi(x),
        & x\in (0,1).
    \end{dcases}
\end{equation*}
With the same discretization technique adopted in Section \ref{sec:discre}, we obtain the linear system
\begin{equation*}
    \tilde{A}_{n,m}\mathbf{u} = \mathbf{b},
\end{equation*}
where
\begin{gather*}
    \renewcommand{\arraystretch}{1.5}
    \tilde{A}_{n,m} := \left[ \begin{array}{cccc}
        \gamma_0 I_m + \frac{\alpha_n}{\beta_m} G_m & \mathrm{O}_m & \cdots & \mathrm{O}_m \\
        \gamma_1 e^{-\rho \Delta t} I_m & \gamma_0 I_m + \frac{\alpha_n}{\beta_m} G_m & \ddots & \vdots \\
        \vdots & \ddots & \ddots & \mathrm{O}_m \\
        \gamma_{n-1} e^{-(n-1)\rho\Delta t} I_m & \cdots & \gamma_1 e^{-\rho \Delta t} I_m & \gamma_0 I_m + \frac{\alpha_n}{\beta_m} G_m
    \end{array} \right]
    \in \C^{nm\times nm},
    \\
    \renewcommand{\arraystretch}{1.5}
    \mathbf{u} :=
        \left[ \begin{array}{c}
        \mathbf{u}^{(1)} \\
        \mathbf{u}^{(2)} \\
        \vdots \\
        \mathbf{u}^{(n)}
    \end{array} \right]
    \in\C^{nm},
    \qquad
    \mathbf{b} :=
        \left[ \begin{array}{c}
            b_0 e^{-\rho\Delta t} \boldsymbol{\varphi} + \alpha_n q^{(1)} \mathbf{f} \\
            b_1 e^{-2\rho\Delta t} \boldsymbol{\varphi} + \alpha_n q^{(2)} \mathbf{f} \\
            \vdots \\
            b_{n-1} e^{-n\rho\Delta t} \boldsymbol{\varphi} + \alpha_n q^{(n)} \mathbf{f}
        \end{array} \right]
        \in\C^{nm}.
\end{gather*}
By solving this system, we construct final time data $\mathbf{u}^{(n)}$, which we perturb by adding a random quantity $\varepsilon\delta (x)$, where $\delta (x)$ is randomly selected in the interval $(-1,1)$. In this way, we simulate the contamination of $\phi(x)$ by noise, obtaining $\boldsymbol{\phi}_\varepsilon$.

Then, we apply the GMRES method to \eqref{eq:lin-syst} and report in Table \ref{table:gmres} the number of iterations required by both the nonpreconditioned and preconditioned linear systems. We set the regularization parameter to $\lambda=5\cdot10^{-3}$, the noise level to $\varepsilon=0.01$, and vary both fractional orders. The initial guess is the zero vector, the tolerance is $\texttt{tol}=10^{-8}$, and the maximum number of iterations is equal to the size of the matrix, $N=(n+1)m$.

Comparing the columns corresponding to the preconditioned and nonpreconditioned linear system, we observe that, for each choice of the fractional orders, $P_N$ is indeed effective in reducing the number of iterations, confirming the theoretical predictions. It also appears to be stable, maintaining good results as the problems sizes increase and across all choices of fractional orders. Indeed, $P_N$ demonstrates the best performance when the iteration counts for the nonpreconditioned linear system are particularly high, i.e, for $\xi=0.2$, $\eta=0.8$.

Furthermore, for each choice of $\xi$ and $\eta$, the convergence rate in the preconditioned framework seems to depend almost exclusively on the spatial discretization finesse: in fact, for each fixed $m$, the number of iterations does not vary significantly as we increase $n$. This is reasonable, considering that the number of outliers in the cluster at 1 of the preconditioned matrix depends only on $m$.

Furthermore, looking more in detail, we observe that the growth in the parameter $m$ is less than linear, i.e., less than a constant times the square root of the global size, when $m$ and $n$ grow proportionally. The described behaviour agrees with the negative results in \cite{MultiNo,MultiNo2}, where theoretical barriers to the optimality of preconditioned Krylov solvers are proven.

\begin{table}[ht] 
    \centering
    \caption{Iterations required for solving the linear system with the GMRES method, without and with the preconditioner, for various choices of the fractional orders.}
    \label{table:gmres}
    \begin{tabular}{cccccccc}
    \toprule
    & & \multicolumn{2}{c}{$\xi=0.2$} & \multicolumn{2}{c}{$\xi=0.5$} & \multicolumn{2}{c}{$\xi=0.8$} \\
    & & \multicolumn{2}{c}{$\eta=0.8$} & \multicolumn{2}{c}{$\eta=0.5$} & \multicolumn{2}{c}{$\eta=0.2$} \\
    \cmidrule(lr){3-4}\cmidrule(lr){5-6}\cmidrule(lr){7-8}
    $m$ & $n$ & - & $P_N$ & - & $P_N$ & - & $P_N$ \\
\midrule
$2^{4}$ & $2^{4}$ & 66 & 15 & 53 & 17 & 45 & 16 \\
$2^{4}$ & $2^{5}$ & 62 & 16 & 60 & 17 & 63 & 16 \\
$2^{4}$ & $2^{6}$ & 64 & 17 & 67 & 17 & 86 & 17 \\
$2^{4}$ & $2^{7}$ & 61 & 17 & 67 & 17 & 141 & 17 \\
$2^{4}$ & $2^{8}$ & 61 & 17 & 79 & 17 & 230 & 16 \\
\midrule
$2^{5}$ & $2^{4}$ & 107 & 26 & 75 & 30 & 49 & 28 \\
$2^{5}$ & $2^{5}$ & 110 & 27 & 79 & 29 & 68 & 27 \\
$2^{5}$ & $2^{6}$ & 110 & 27 & 86 & 29 & 90 & 27 \\
$2^{5}$ & $2^{7}$ & 103 & 27 & 89 & 29 & 142 & 27 \\
$2^{5}$ & $2^{8}$ & 100 & 27 & 96 & 30 & 230 & 28 \\
\midrule
$2^{6}$ & $2^{4}$ & 213 & 36 & 101 & 49 & 51 & 40 \\
$2^{6}$ & $2^{5}$ & 211 & 35 & 103 & 48 & 71 & 39 \\
$2^{6}$ & $2^{6}$ & 207 & 36 & 110 & 47 & 94 & 39 \\
$2^{6}$ & $2^{7}$ & 193 & 37 & 109 & 48 & 145 & 41 \\
$2^{6}$ & $2^{8}$ & 194 & 36 & 117 & 49 & 231 & 40 \\
\midrule
$2^{7}$ & $2^{4}$ & 503 & 46 & 136 & 71 & 56 & 57 \\
$2^{7}$ & $2^{5}$ & 493 & 46 & 136 & 71 & 76 & 56 \\
$2^{7}$ & $2^{6}$ & 458 & 46 & 137 & 70 & 98 & 56 \\
$2^{7}$ & $2^{7}$ & 440 & 46 & 136 & 72 & 146 & 56 \\
$2^{7}$ & $2^{8}$ & 436 & 48 & 138 & 71 & 232 & 56 \\
\midrule
$2^{8}$ & $2^{4}$ & 1237 & 60 & 181 & 101 & 59 & 77 \\
$2^{8}$ & $2^{5}$ & 1165 & 59 & 174 & 100 & 78 & 76 \\
$2^{8}$ & $2^{6}$ & 1149 & 58 & 171 & 99 & 100 & 76 \\
$2^{8}$ & $2^{7}$ & 1044 & 59 & 168 & 100 & 150 & 76 \\
$2^{8}$ & $2^{8}$ & 985 & 60 & 166 & 101 & 234 & 76 \\
    \bottomrule
    \end{tabular}
\end{table}

\section{Conclusions} \label{sec:end}
In this work, we analyzed the inverse problem of reconstructing the source term in a time-space fractional diffusion equation. Since the problem is inherently ill-posed, the quasi-boundary value regularization method has been employed. We discretized the problem both in space and time via a finite difference scheme, leading to a family of linear system of increasing dimension. The discretization matrices exhibit a 2-level block structure with an inherent lower triangular Toeplitz configuration. We studied the global properties of these sequences of linear systems using the tools provided by GLT theory and were able to provide a preconditioner that yields a weak eigenvalue cluster at 1. The theoretical study was corroborated by numerical experiments, which showed that the proposed preconditioner significantly accelerates the convergence of the GMRES method. Further research may focus on constructing a more accurate preconditioner that accounts for the remaining outliers and possibly leads to a strong eigenvalue cluster at 1. In fact, even in presence of the negative findings in \cite{MultiNo,MultiNo2}, there exist cases especially in the fractional setting where the optimality can still be reached; see \cite{against barrier1,against barrier2}.

\section*{Acknowledgments}
The research of Stefano Serra-Capizzano is supported by the PRIN-PNRR project \lq\lq MATH-ematical tools for predictive maintenance and PROtection of CULTtural heritage (MATHPROCULT)'' (code P20228HZWR, CUP J53D23003780006), by INdAM-GNCS Project \lq\lq Analisi e applicazioni di matrici strutturate (a blocchi)'' (CUP E53C23001670001), and by the European High-Performance Computing Joint Undertaking (JU) under Grant Agreement 955701. The JU receives support from the European Union’s Horizon 2020 research and innovation programme and Belgium, France, Germany, Switzerland. Furthermore Stefano Serra-Capizzano is grateful for the support of the Laboratory of Theory, Economics and Systems – Department of Computer Science at Athens University of Economics and Business. The research of Rosita L. Sormani is funded by the PRIN-PNRR project \lq\lq A mathematical approach to
inverse problems arising in cultural heritage preservation and dissemination'' (code P2022PMEN2, CUP F53D23010100001). Finally Stefano Serra-Capizzano and Rosita L. Sormani are partly supported by Italian National Agency INdAM-GNCS.


\begin{thebibliography}{99}


\bibitem{against barrier2}
N. Barakitis, S.-E. Ekström, and P. Vassalos, \emph{Preconditioners for fractional diffusion equations based on the spectral symbol}, 
Numerical Linear Algebra with Applications 29 (5) (2022), e2441.




\bibitem{barbarino2020uni}
G. Barbarino, C. Garoni, and S. Serra-Capizzano, \emph{Block generalized locally Toeplitz sequences: theory and applications in the unidimensional case}, Electronic Transactions on Numerical Analysis 53 (2020), pp. 28–112.


\bibitem{barbarino2020multi} G. Barbarino, C. Garoni, and S. Serra-Capizzano, \emph{Block generalized locally Toeplitz sequences: theory and applications in the multidimensional case}, Electronic Transactions on Numerical Analysis 53 (2020), pp. 113–216.




\bibitem{Bu2014}
W. Bu, Y. Tang, and J. Yang, \emph{Galerkin finite element method for two-dimensional Riesz space fractional diffusion equations}, Journal of Computational Physics 276 (2014), pp. 26-38.


\bibitem{Chen2014}
M. Chen and W. Deng, \emph{Fourth order accurate scheme for the space fractional diffusion equations}, SIAM Journal on Numerical Analysis 52 (3) (2014), pp. 1418-1438.


\bibitem{Denche2005}
M. Denche and K. Bessila, \emph{A modified quasi-boundary value method for ill-posed problems}, Journal of Mathematical Analysis and Applications 301 (2) (2005), pp. 419-426.


\bibitem{Dhatt2012}
G. Dhatt, E. Lefrançois, and G. Touzot, \emph{Finite element method} John Wiley \& Sons, 2012.


\bibitem{Donatelli2016}
M. Donatelli, M. Mazza, and S. Serra-Capizzano, \emph{Spectral analysis and structure preserving preconditioners for fractional diffusion equations}, Journal of Computational Physics 307 (2016), pp. 262-279.


\bibitem{Eymard2000}
R. Eymard, T. Gallouët, and R. Herbin, \emph{Finite volume methods}, Handbook of numerical analysis, 7 (2000), pp. 713-1018.


\bibitem{garoni2017} C. Garoni and S. Serra-Capizzano, \emph{Generalized locally Toeplitz sequences: theory and applications. Vol. I}, Springer, Cham, 2017.


\bibitem{garoni2018} C. Garoni and S. Serra-Capizzano, \emph{Generalized locally Toeplitz sequences: theory and applications. Vol. II}, Springer, Cham, 2018.

\bibitem{glt-tutorial}
C. Garoni and S. Serra-Capizzano, \emph{Generalized locally Toeplitz sequences: a spectral analysis tool for discretized differential equations}, Splines and PDEs: from approximation theory to numerical linear algebra, 161–236, Lecture Notes in Math., 2219, Fond. CIME/CIME Found. Subser., Springer, Cham, 2018.

\bibitem{tom} C. Garoni, H. Speleers, S.-E. Ekström, A. Reali, S. Serra-Capizzano, and T.J.R. Hughes,  \emph{Symbol-based analysis of finite element and isogeometric B-spline discretizations of eigenvalue problems: exposition and review}, Archives of Computational Methods in Engineering. State of the Art Reviews  26 (5) (2019), pp. 1639–1690.


\bibitem{Golinskii2007}
L. Golinskii and S. Serra-Capizzano, \emph{The asymptotic properties of the spectrum of nonsymmetrically perturbed Jacobi matrix sequences}, Journal of Approximation Theory 144 (1) (2007), pp. 84-102.


\bibitem{Gottlieb1977}
D. Gottlieb and S.A. Orszag, \emph{Numerical analysis of spectral methods: theory and applications}, Society for Industrial and Applied Mathematics, Philadelphia, 1977.


\bibitem{Hao2021}
Z. Hao, Z. Zhang, and R. Du, \emph{Fractional centered difference scheme for high-dimensional integral fractional Laplacian}, Journal of Computational Physics 424 (2021), 109851.


\bibitem{Hao2008}
D.N. Hào, N.V. Duc, and H. Sahli, \emph{A non-local boundary value problem method for parabolic equations backward in time}, Journal of mathematical analysis and applications 345 (2) (2008), pp. 805-815.




\bibitem{Rooh2024}
A. Ilyas, R.A. Khalid, and S.A. Malik, \emph{Identifying temperature distribution and source term for generalized diffusion equation with arbitrary memory kernel}, Mathematical Methods in the Applied Sciences 47 (7) (2024), pp. 5894-5915.


\bibitem{Malik2024}
A. Ilyas and S.A. Malik, \emph{Direct and some inverse problems for a generalized diffusion equation with variable coefficients}, Computational and Applied Mathematics 43 (6) (2024), 364.
	

\bibitem{Ilyas2024} A. Ilyas and S. Serra-Capizzano, \emph{Inverse source problems for identifying time and space-dependent coefficients in a 2D generalized diffusion equation}, Applied Mathematics and Computation {507} (2025), 129597.


\bibitem{Jin2015}
B. Jin and W. Rundell, \emph{A tutorial on inverse problems for anomalous diffusion processes}, Inverse problems 31 (3) (2015), 035003.






\bibitem{Kilbas2006}
A. Kilbas, H.M. Srivastava, J.J. Trujillo, \emph{Theory and applications of fractional differential equations, vol. 204}, Elsevier, Amsterdam, 2006.


\bibitem{Lei2013}
S.-L. Lei and H.-W. Sun, \emph{A circulant preconditioner for fractional diffusion equations}, Journal of Computational Physics 242 (2013), pp. 715-725.


\bibitem{Li2018}
M. Li, X.-M. Gu, C. Huang, M. Fei, and G. Zhang, \emph{A fast linearized conservative finite element method for the strongly coupled nonlinear fractional Schrödinger equations}, Journal of Computational Physics 358 (2018), pp. 256-282.




\bibitem{Magin2010}
R.L. Magin, \emph{Fractional calculus models of complex dynamics in biological tissues}, Computers \& Mathematics with Applications 59 (5) (2010), pp. 1586-1593.












\bibitem{MultiNo2}
D. Noutsos, S. Serra-Capizzano, and P. Vassalos, \emph{Matrix algebra preconditioners for multilevel Toeplitz systems do not insure optimal convergence rate}, Theoretical Computer Science 315 (2-3) (2004), pp. 557–579.

\bibitem{against barrier1}
D. Noutsos, S. Serra-Capizzano, and P. Vassalos, \emph{Essential spectral equivalence via multiple step preconditioning and applications to ill conditioned Toeplitz matrices}, Linear Algebra and its Applications 491 (2016), pp. 276–291.


\bibitem{Owolabi2024}
K.M. Owolabi, E. Pindza, B. Karaagac, and G. Oguz, \emph{Laplace transform-homotopy perturbation method for fractional time diffusive predator–prey models in ecology}, Partial Differential Equations in Applied Mathematics 9 (2024), 100607.


\bibitem{Pang2021}
H.-K. Pang, H.-H. Qin, H.-W. Sun, and T.-T. Ma, \emph{Circulant-based approximate inverse preconditioners for a class of fractional diffusion equations}, Computers \& Mathematics with Applications 85 (2021), pp. 18-29.



\bibitem{Saad1986GMRES}
Y. Saad, and M.H. Schultz, \emph{GMRES: A generalized minimal residual algorithm for solving nonsymmetric linear systems}, {SIAM J. Sci. Stat. Comput.} {7} (3) (1986), pp. 856-869. 



\bibitem{Sandev2017}
T. Sandev, I.M. Sokolov, R. Metzler, and A. Chechkin, \emph{Beyond monofractional kinetics}, Chaos, Solitons \& Fractals 102 (2017), pp. 210–217.


\bibitem{Scalas2000}
E. Scalas, R. Gorenflo, and F. Mainardi, \emph{Fractional calculus and continuous-time finance}, Physica A: Statistical Mechanics and its Applications 284 (1-4) (2000), pp. 376-384.


\bibitem{Scher2002}
H. Scher, G. Margolin, R. Metzler, J. Klafter, and B. Berkowitz, \emph{The dynamical foundation of fractal stream chemistry: The origin of extremely long retention times}, Geophysical Research Letters 29 (5) (2002), pp. 5-1--5-4.


 
\bibitem{MultiNo}
S. Serra Capizzano, \emph{Matrix algebra preconditioners for multilevel Toeplitz matrices are not superlinear},  Linear Algebra and its Applications 343–344 (2002), pp. 303-319.




    




\bibitem{Simmons2017}
A. Simmons, Q. Yang, and T. Moroney, \emph{A finite volume method for two-sided fractional diffusion equations on non-uniform meshes}, Journal of Computational Physics 335 (2017), pp. 747-759.




\bibitem{Suhaib2023}
K. Suhaib, A. Ilyas, and S.A. Malik, \emph{On the inverse problems for a family of integro-differential equations}, Mathematical Modelling and Analysis 28 (2) (2023), pp. 255-270.

\bibitem{Sun2020}
Z. Sun, G. Gao, \emph{Fractional differential equations - finite difference methods}, De Gruyter, Berlin, 2020.




\bibitem{Thomas2013}
J.W. Thomas, \emph{Numerical partial differential equations: finite difference methods}. Vol. 22. Springer Science \& Business Media, 2013.



\bibitem{Tuan2017}
N.H. Tuan, L.D. Long, V.T. Nguyen, and T. Tran, \emph{On a final value problem for the time-fractional diffusion equation with inhomogeneous source}, Inverse Problems in Science and Engineering 25 (9) (2017), pp. 1367-1395.






\bibitem{Wang2011}
Z. Wang, X. Huang, and G. Shi, \emph{Analysis of nonlinear dynamics and chaos in a fractional order financial system with time delay}, Computers \& Mathematics with Applications 62 (3) (2011), pp. 1531-1539.








\bibitem{Zhou2022}
F.X. Zhou, L.Y. Wang, Z.Y. Liu and W.C. Zhao. \emph{A viscoelastic-viscoplastic mechanical model of time-dependent materials based on variable-order fractional derivative}, Mechanics of Time-Dependent Materials 26, (3) (2022), pp. 699-717.









\end{thebibliography}
\end{document}